\documentclass[a4paper]{amsart} 

\usepackage{geometry}
\usepackage[english]{babel}
\usepackage[utf8]{inputenc}

\usepackage{amssymb,amsmath,amsfonts,amsthm}
\usepackage{verbatim,url}
\usepackage{setspace}
\usepackage{hyperref}

\usepackage{tikz}
\usepackage{tikz-cd}
\usepackage[all,cmtip]{xy}
\usepackage{pgfplots}
\usepackage{verbatim}
\usepackage{longtable}
\usepackage{esvect}
\usetikzlibrary{math} 

\newtheorem{theorem}{Theorem}[section]

\newtheorem{Lemma}[theorem]{Lemma}

\newtheorem{remark}[theorem]{Remark}

\theoremstyle{remark}

\newtheorem{rem}[theorem]{Remark}
\newtheorem{Example}[theorem]{Example}

\newcommand{\Aut}{\operatorname{Aut}}

\begin{document}

\title{Some surfaces with canonical maps of degree $10$, $11$ and $14$}
\author{ Federico Fallucca and Christian Glei\ss ner}
%\address{Federico Fallucca and Christian Gleissner  \newline University of Bayreuth, Universit\"atsstr. 30, D-95447 Bayreuth, Germany }
%\email{christian.gleissner@uni-bayreuth.de,Â  federico.fallucca@uni-bayreuth.de}

\thanks{
\textit{Acknowledgements:} The authors would like to thank  Roberto Pignatelli and Massimiliano Alessandro
 for useful comments and discussions. A special thanks goes to the organizers of the conference  \grqq Algabraic Geometry in Roma Tre, on the occasion of Sandro Verra's 70th birthday\grqq (June 2022), that enabled us to participate, present this work as a poster and provided financial support for our stay in Rome.}
\address{Federico Fallucca 
\newline Università di Trento, Via Sommarive 14, I-38123 Trento (TN), Italia}
\email{federico.fallucca@unitn.it}
\address{Christian Gleissner 
\newline University of Bayreuth, Universit\"atsstr. 30, D-95447 Bayreuth, Germany}
\email{Christian.Gleissner@uni-bayreuth.de}

\begin{abstract}
In this note we present examples of complex algebraic surfaces of general type with canonical maps of degree $10$, $11$ and $14$. They are constructed as quotients 
of a product of two Fermat septics using certain free actions of the group $\mathbb Z_7^2$.
\end{abstract}

\maketitle

\section{Introduction}

\noindent
Beauville has shown in \cite{B79} that if the image of the canonical map $\Phi_{K_S}$ of a surface $S$ of general type is a surface, then the following inequality holds:  
\[
d:=\deg(\Phi_{K_S}) \leq 9+\frac{27-9q}{p_g-2}\leq 36.
\]
Here,  $q$ is the irregularity and $p_g$ the geometric genus of $S$. 
In particular, $28 \leq d$ is only possible if $q=0$ and $p_g=3$.
%Beauville has shown in \cite{B79} that the degree $d$ of the canonical map of a surface $S$ of general type is at most  $36$. The maximum is reached if and only if $S$ is regular,  $p_g(S)=3$, $K_S^2=36$ and
 %the canonical system $\vert K_S\vert$ is base-point free. 
Motivated  by this observation, 
the construction  of surfaces with $p_g=3$ and canonical map of degree $d$ for every  value $2 \leq d \leq  36$ is an 
interesting, but still widely open  problem \cite[Question 5.2]{MLP21}. In particular for most values   $10 \leq d$, it is not known if a surface realising that degree  exists. 
Indeed, for a long time, the only examples were the surfaces of 
Persson \cite{Per},  with canonical map of degree $16$ and Tan \cite{Tan}, with degree $12$.  
In recent years, this
problem attracted the attention of  many authors,  putting an increased effort in the construction of new examples. 
 As a result,  previously unknown surfaces with degree $d=12,16, 20, 24, 32$ and  $36$ 
have been discovered, by  Rito  \cite{Ri15, Ri17, Ri17Zwei, Ri19}, Gleissner, Pignatelli and Rito \cite{GPR} and  Nguyen \cite{Bin19, Bin21}. 
In this work, we present surfaces  with canonical maps of degree $d=10$, $11$ and $14$. According to our knowledge they are the first surfaces for those degrees. 
They  can be described using Pardini's theory of branched abelian covers \cite{Pa91}, which is one of the standard techniques in this subject, 
cf.  \cite{MLP21}. However, we decided to present them 
in an elementary way using plane curves and basic algebraic geometry at graduate textbook level \cite{Shafa}. Our construction is completely self-contained, basically reference free and easily  
accessible. It can be sketched  as follows:   the surfaces $S$, which all have $p_g=3$, arise  as quotients of a product of two Fermat septics 
\[
F= \lbrace x_0^7+x_1^7+x_2^7=0\rbrace \subset \mathbb P^2
\]
 modulo certain free and diagonal  actions of the group $\mathbb Z_7^2$. 
 %Our surfaces are  regular with geometric genus $p_g=3$. 
 Their explicit construction allows us to 
write the canonical system of each of them in terms of three $\mathbb Z_7^2$-invariant holomorphic two-forms on the product $F\times F$. 
It turns out that for none of them  $\vert K_{S}\vert$ is base-point free, i.e. the canonical map 
$\Phi_{K_{S}} \colon S \dashrightarrow \mathbb P^2$ is just a rational map. To compute its degree, we resolve
 the indeterminacy by a sequence of blowups and compute  the degree of the resulting morphism  via elementary intersection theory.  \\
We point out that our surfaces are particular examples of surfaces isogenous to a product, i.e. quotients of a product of two curves modulo a free  action of a finite group.
This construction goes back to an exercise in  Beauville's book on \emph{Complex Algebraic Surfaces} \cite[Exercises X.13 4]{Beauville},  where a free action of $\mathbb Z_5^2$  on a product of two Fermat quintics is used to 
construct a surface with 
 $p_g=q=0$. This example  served  as an important inspiration for our work.
%also encouraged the reader to find more of these examples. 
By now, many more surfaces isogenous to a product  have been constructed.  Apart from other works, that mainly deal with irregular surfaces, we want to mention the 
complete classification of surfaces isogenous to a product with $p_g=q=0$ \cite{BCG} and the classification  for $p_g=1$ and  $q=0$ under the assumption that the action is diagonal  \cite{G15}.  
However, for higher values of $p_g$,   a classification of regular surfaces isogenous to a product is  much more involved and is  not yet established.  
Recently similar constructions involving non-free actions on a product of Fermat curves have been  used to provide 
other interesting projective manifolds that helped us to understand some important geometric phenomena.  Most notably 
are the rigid but not infinitesimally rigid manifolds \cite{BP21} of Bauer and Pignatelli that gave a  negative answer to  a question of Kodaira and Morrow \cite[p.45]{KM71} and, to a lesser degree, 
also the infinite series
 of $n$-dimensional infinitesimally rigid manifolds of general type with non-contractible universal cover for each $n\geq 3$,  provided  by 
 Frapporti and the second author of this paper \cite{FG}. %For $n=2$, the existence of such manifolds  is still an  open question, see  \cite{BC}.  

%, and obtained as quotients by an  abelian group. 

%Under these assumptions, a classification can be carried out: there are precisely seven isomorphism classes. All of the seven surfaces can be realised as a product of two Fermat septics modulo a suitable action of $\mathbb{Z}_7^2$. 
%by using a modified version of the MAGMA algorithm \cite{GPR} we could 
\medskip
\noindent
{\bf Notation:} Throughout the paper a surface $S$ is a  projective manifold of dimension two. 
We use standard terminology in surface theory, cf. \cite{B79}.

\section{The surfaces}

\noindent
In this section we construct a series of surfaces $S$,  
 as quotients of a product of two Fermat septics $F$, 
modulo a suitable diagonal  action of the group $\mathbb Z_7^2$. For any surface $S$, 
we determine the canonical map $\Phi_{K_S}$ and compute its  degree. 
 
\noindent 
On the first copy of $F$ we define the action of $\mathbb Z_7^2$ as 
\[
\phi \colon \mathbb Z_7^2 \to \Aut(F), \quad (a,b) \mapsto [(x_0:x_1:x_2) \mapsto (x_0: \zeta_7^a x_1:  \zeta_7^b x_2)]. 
\]
This action has   $21$ points  with non trivial stabilizer. They form three orbits of length $7$. A representative of each orbit and a generator of the stabilizer is given by: 
\[
\begin{tabular}{c | c | c | c }
\makebox{point} & $(-1:0:\zeta_7)$ & $(-1:\zeta_7:0)$  & $ (0: -1 :\zeta_7)$   \\
\hline
\makebox{generator} & $ (1,0) $ & $(0,1) $  & $(6,6) $  \\
\end{tabular}
\]
Note that the automorphisms $\phi(a,b)$ are precisely the   deck transformations of the  cover
\[
\pi \colon F \to \mathbb P^1, \qquad (x_0:x_1:x_2) \mapsto (x_1^7:x_2^7).
\]
The cover  has  degree $49$ and  is branched along $(0:1)$, $(1:0)$ and $(-1:1)$. In particular 
$F/\mathbb Z_7^2 \simeq \mathbb P^1$ and $\pi$ is the quotient map. 

\noindent 
On the second copy of $F$, for which we use the homogenous variables $y=(y_0:y_1:y_2)$, the group   acts  by $\phi\circ A$, where $A\in \Aut(\mathbb Z_7^2)$ is an automorphism depending on the specific example. 
The explicit choices for $A$ are stated in the tables below.
To write the canonical systems of the corresponding unmixed quotients 
\[
S:=(F\times F)/\mathbb Z_7^2 \qquad \makebox{modulo the diagonal actions} \qquad \phi \times (\phi \circ A),
\]
we need to fix a suitable 
basis of the space $H^0(F,\Omega_F^1)$ of \emph{global holomorphic 1-forms} on $F$. In affine coordinates  such a basis is given by
\[
\lbrace \omega_{jk}:=u^jv^{k-6} du ~ \vert ~ j+k \leq 4\rbrace, \qquad \makebox{where} \qquad u:=\frac{x_1}{x_0} \quad \makebox{and} \quad   v:=\frac{x_2}{x_0}.
\]
Note that:
\begin{itemize}
\item[I)]
The action of $\mathbb Z_7^2$ on $H^0(F,\Omega_F^1)$ under pullback with $\phi$   is 
\[
 \phi(a,b)^{\ast}(\omega_{jk})=\zeta_7^{a(j+1)+b(k-6)}\omega_{jk}. 
\]
\item[II)] 
The space of canonical sections $H^0(K_{S})$ is isomorphic to 
$\big(H^0(\Omega_{F}^1) \otimes H^0(\Omega_{F}^1) \big)^{\mathbb Z_7^2}$, 
where  the action on the tensor product is diagonal, i.e.  $(a,b)\in \mathbb Z_7^2$ acts via 
\[
\phi(a,b)^{\ast} \otimes  \phi\big(A(a,b)\big)^{\ast}. 
\]
\end{itemize}

\noindent 
The observations I) and II) imply: 

\begin{Lemma}\label{congruence}
A basis of $H^0(K_S)$ is given by the $\mathbb Z_7^2$-invariant tensors $\omega_{jklm}:=\omega_{jk}\otimes \omega_{lm}$. A tensor $\omega_{jklm}$ is invariant if and only if 
for all $(a,b)\in \mathbb Z_7^2$  it holds: 
 \[
a(j+1)+b(k-6) + a'(l+1)+b'(m-6)
\equiv 0 ~ mod ~7,  \quad \makebox{where} \quad \begin{pmatrix} a' \\ b' \end{pmatrix}:=A
\begin{pmatrix} a \\ b \end{pmatrix}.
\]
\end{Lemma}

\noindent 
We can now state and prove our main result:

\begin{theorem}\label{MainTheo}
For all  $A\in \Aut(\mathbb Z_7^2)$ in  the table  below, 
the diagonal action $\phi \times (\phi \circ A)$ of $\mathbb Z_7^2$ on the product of two Fermat septics is free. The quotient is a regular smooth projective surface $S$ of general type with $p_g=3$.  
A basis of $H^0(K_S)$, the canonical map $\Phi_{K_S}$ in projective coordinates and its degree are stated  in 
the table:

\[
\begin{tabular}{c | c | c | c | c }
\makebox{No} & $A$ & \makebox{Basis of $H^0(K_S)$}  & $\Phi_{K_S}(x,y)$ & $\deg(\Phi_{K_S})$  \\
\hline
1. & $ \begin{pmatrix} 4&5 \\ 3 & 1\end{pmatrix} $ & $\lbrace \omega_{1304},\omega_{2210}, \omega_{3012}  \rbrace$ &   
$(x_1x_2^3y_2^4: x_1^2x_2^2y_0^3y_1: x_0x_1^3y_0y_1y_2^2)$ & $10$ \\ 
\hline
2. & $ \begin{pmatrix} 2&6 \\ 1 & 4\end{pmatrix} $ & $\lbrace    
\omega_{0011},
\omega_{1202},
\omega_{2040} \rbrace$ &  
$(x_0^4y_0^2y_1y_2:x_0x_1x_2^2y_0^2y_2^2: x_0^2x_1^2y_1^4)$ & $11$ \\
\hline
3. & $ \begin{pmatrix} 3&3 \\ 6 & 4\end{pmatrix} $ & 
$\lbrace   
\omega_{0103},
\omega_{1310}, 
\omega_{3031} \rbrace$  &   
$(x_0^3x_2y_0y_2^3:x_1x_2^3y_0^3y_1:x_0x_1^3y_1^3y_2)$  & $14$  \\ 
\hline
\end{tabular}
\]
\end{theorem}

\begin{proof}
First we show that 
the three diagonal actions $\phi \times (\phi \circ A)$  on $F\times F$ are free. Indeed, 
as remarked above, the non-trivial stabilizers of the points on the first copy of $F$ are  generated by $(1,0)$, $(0,1)$ and $(6,6)$. However, none of these elements have a fixed point on the second copy of $F$ 
under the twisted actions $\phi \circ A$. 
Thus, the actions are free and the quotient surfaces $S$ are smooth, projective and of general type. The latter holds because the genus of the Fermat septic is  $g(F)=15 \geq 2$.
Moreover, they are regular surfaces, because they do not possess any non-zero holomorphic one-forms, since  $F/\mathbb Z_7^2$ is biholomorphic to $\mathbb P^1$. 
The geometric genus of each  $S$ is therefore equal to 
\[
p_g=  \chi(\mathcal O_{S})- 1 = \frac{(g(F)-1)^2}{\vert \mathbb Z_7^2\vert}-1=\frac{14^2}{49}-1=3. 
\]
Using Lemma  \ref{congruence}, we compute a basis of $H^0(K_S)$ for each surface $S$. 
Replacing the affine variables  by $\frac{x_i}{x_0}$ and $\frac{y_j}{y_0}$ and multiplying by $x_0^4y_0^4$ we obtain the 
 bi-quartics that define the canonical map.  \\
It remains to determine the degree of $\Phi_{K_S}$ for each surface $S$. For this purpose we 
\emph{resolve  the indeterminacy} of these maps 
by a sequence of  blowups, as explained in the textbook \cite[Theorem II.7]{Beauville}:

\[
\xymatrix{
\hat{S} \ar[r] \ar[dr]_{\Phi_{\hat{M}}} & S\ar@{-->}[d]^{\Phi_{K_S}} \\
& \mathbb P^2.
}
\]
Here, $\vert \hat{M} \vert$ is a base-point free linear system.  
The     self-intersection $\hat{M}^2$ is positive if and only if 
$\Phi_{\hat{M}}$ is not composed by a pencil. In this case $\Phi_{\hat{M}}$  is onto and  
 it holds: 
\[
\deg(\Phi_{K_S})=\deg(\Phi_{\hat{M}})=\hat{M}^2. 
\]
For the computation of the resolution,  it is convenient to write the divisors of the  bi-quartics defining  $\Phi_{K_S}$  as linear combinations of the 
reduced curves  $F_j:=\lbrace x_j=0\rbrace$ and $G_k:=\lbrace y_k=0\rbrace$ on $S$. 
Note that $F_j$ and  $G_k$ intersect transversally in only one point  and  $(F_j,F_k)=(G_j,G_k)=0$, for all  $j,k$.  
Thus, these curves  can be illustrated as a grid of three vertical and three horizontal lines.  \\
As an example, consider the first surface in the table. 
Here, the divisors of the three bi-quartics spanning  the canonical system  $\vert K_{S}\vert$ are: 
\[
F_1+ 3F_2+4G_2, \qquad 
2F_1+2F_2+3G_0+G_1 \qquad \makebox{and} \qquad 
F_0+3F_1+G_0+G_1+2G_2. 
\]
The fixed part of $\vert K_S\vert$ is $F_1$ and the mobile part $\vert M \vert $  has precisely four base-points: 
\[
F_1\cap G_2= \lbrace p_{12} \rbrace, \quad
F_2\cap G_0 = \lbrace p_{20}\rbrace, \quad 
F_2\cap G_1 = \lbrace p_{21} \rbrace \quad \makebox{and} \quad 
F_2\cap G_2= \lbrace p_{22} \rbrace.
\]
We blow up these points, take the pullback of the mobile part $\vert M \vert$  of $\vert K_{S}\vert$
and remove the fixed part of this new linear system. We repeat the procedure, until we obtain a  base-point free  linear system 
$\vert\hat{M}\vert $. The degree of the canonical map can then be  computed as  
$\deg(\Phi_{K_S})=\hat{M}^2$.  
Alternatively, we can shortcut the computation by  using Lemma  \ref{FedericoLemma} from below, which tells us  the 
 contribution of the difference $M^2- \hat{M}^2$ coming from each base-point simply by looking at the coefficients of the divisors 
  spanning   the  mobile part 
 of $\vert K_S\vert$.  See Example \ref{FedericoExample}  for an illustration. \\
Fortunately  the conditions of Lemma  \ref{FedericoLemma} on the coefficients of the divisors are fulfilled also for any other surface of the table, providing us with an easier way to compute the degree of the canonical map.

 % Alternatively, we may apply Lemma  \ref{FedericoLemma} from below, which tells us  the 
% contribution of the difference $M^2- \hat{M}^2$ coming from each base-point simply by looking at the coefficients of the divisors 
%  spanning   the  mobile part 
% of $\vert K_S\vert$.  See Example \ref{FedericoExample}  for an illustration. 
% Lemma \ref{FedericoLemma} can be easily applied for any surface of the table to compute directly the contribution of the difference $M^2- \hat{M}^2$ and so also
% the degree of the canonical map $\deg(\Phi_{K_S})=\hat{M}^2$.  
\end{proof}

\begin{Lemma}\label{FedericoLemma} 
Let  $\vert M \vert $ be a two-dimensional linear system on a surface $S$, with only isolated base-points, which is spanned by $D_1$, $D_2$ and $D_3$. Assume that in a neighborhood of a basepoint $p$, we can write  the divisors 
$D_i$ as 
\[
D_1=aH, \quad D_2=bK \quad \makebox{and} \quad D_3= cH+d K, 
\]
where $H$ and $K$ are reduced, smooth and intersect transversally at $p$ and $a,b,c,d$ are non-negative integers, $b\leq a$.
Assume that 
\begin{itemize}
    \item $d\geq b$ or
    \item $b\neq 0$ and $c+md\geq a$, where $a=mb+q$ with $0\leq q<b$.
\end{itemize}
Then after blowing up at most $(ab)$-times we obtain a new linear system $\vert \hat{M} \vert $ such that no infinitely near point of 
$p$ is a base-point of $\vert \hat{M} \vert $. Moreover $\hat{M}^2 =M^2-ab$.

%Let $b\leq a$ and write 
%$a=mb+q$ with $0\leq q<b$. 
%If $c+md\geq a $ or $d\geq b$, then 
%after blowing up at most $(ab)$-times we obtain a new linear system $\vert \hat{M} \vert $ such that no infinitely near point of 
%$p$ is a base-point of $\vert \hat{M} \vert $. Moreover $\hat{M}^2 =M^2-ab$.
\end{Lemma}
\begin{proof}
  We prove the lemma by induction on $(a,b)$ with $b\leq a$. Here we are considering the lexicographic order $\leq$ defined on the lower diagonal $\Delta^\geq:=\{(a,b) \colon a\geq b\}\subseteq \mathbb{N}\times \mathbb{N}$ as follows:
  \[
  (a',b')\leq (a,b) \text{\ if \ and \ only \ if \ }  a'< a  \text{ \ or \ } a'=a \text{ \ and \ } b'\leq b.
  \]
  In this case $\Delta^\geq$ admits the \textit{well-ordering principle} and so the \textit{principle of  mathematical induction} holds. \\
  Suppose that $(a,b)=0$. Then $\vert M\vert$ is base-point free and so $\hat{M}=M^2=M^2-ab$. Now suppose that the statement is true for $(a',b')<(a,b)$. Our aim is to prove it for $(a,b)$. We blow up the base-point $p$,  take the pullback of the divisors $D_i$ and remove the fixed part, which is the exceptional divisor  $bE$  of the blowup. In fact the pullback of $D_3$ contains $c+d$ times $E$ and $ c+d\geq b$, thanks to the assumptions $c+md\geq a$ or $d\geq b$:
  If $d\geq b$, then $c+d\geq b$, otherwise if $d<b$ and $c+md\geq a$, then 
  \[
  c+d-b\geq c+m(d-b)\geq c+md-a\geq 0.
  \]
  Restricted to  the preimage of our neighborhood of $p$,  these divisors are: 
\[
a\hat{H}+(a-b)E, \qquad  b\hat{K}\qquad \makebox{and} \qquad c\hat{H}+d\hat{K} +(c+d-b)E. 
\]
Here, $\hat{H}$ and $\hat{K}$  are the strict transforms of $H$ and $K$. Let $\vert \hat{M} \vert$ be the linear system generated 
by these three divisors, then $\hat{M}^2=M^2-b^2$. If $a=b$ or $b=0$, then $\vert \hat{M} \vert $ is base-point free and we are done. Otherwise, on the preimage, the linear system $\vert \hat{M} \vert$ has precisely one 
new base-point: the intersection point  of $\hat{K}$ and  $E$. Locally near this point the three divisors spanning $\vert \hat{M} \vert$ are: 
\[
(a-b)E, \qquad  b\hat{K}\qquad \makebox{and} \qquad d\hat{K} +(c+d-b)E.   
\]
We need to distinguish two cases, when $m=1$ or when $m> 1$. In the first case $a-b=q< b$, so that $(b,q)<(a,b)$. We can write $b=m'q+q'$, with $0\leq q' < q$ and define new coefficients $a':=b$, $b':=q$, $c':=d$ and $d':=c+d-b$. Then they fulfill the inductive hypothesis, because:

If $c+d\geq a$, then 
\[
d'=c+d-b\geq a-b=q=b',
\]
else if $d\geq b$, then 
\[
c'+m'd'\geq c'=d\geq b=a'.
\]
By induction, the self-intersection of the new linear system $\hat{M}$ is equal to 
\[
\hat{M}^2=(M^2-b^2)-qb=M^2-ab.
\]
In the case $m>1$, then $b\leq a-b $ and $(a-b,b)<(a,b)$. We define new variables $a':=a-b$, $b':=b$, $c':=c+d-b$ and $d':=d$. Observe that $a'=a-b=(m-1)b'+q$ and we can define $m':=m-1$. They satisfy the inductive hypothesis, because of the estimations:

If $c+md\geq a$, then 
\[
c'+m'd'=c+d-b+(m-1)d=c+md-b\geq a-b=a',
\]
else if $d\geq b$, then $d'\geq b'$.
Hence the self-intersection of the new linear system $\hat{M}$ is equal to 
\[
\hat{M}^2=(M^2-b^2)-(a-b)b=M^2-ab.
\]
\end{proof}
\begin{remark}\label{CorrTerDeg7}
Lemma \ref{FedericoLemma} fails if both of the conditions $c+md\geq a$ and  $d\geq b$ are not satisfied. For example, taking $a=3$, $b=2$ and $c=d=1$ we get that the correction term of $M^2-\hat{M}^2$ is $5$ instead of $6$.
\end{remark}
%\begin{remark}
%The first part of the proof of Lemma  \ref{FedericoLemma} shows that the upper-bound $(ab)$ on the  number of blowups is sharp. Following the proof, we observe that  the exact number of blowups  is the sum of the quotients that appear when we apply the Euclidean algorithm to the integers $a$ and $b$. In other words, taking  the continued fraction
%\[
%\frac{a}{b}=m_1+\cfrac{1}{m_2+\cfrac{1}{m_3+\cdots}},
%\]
%the number of blowups is $\sum_{i}m_i$. In the notation of the lemma $m_1=m$ and $m_2=m'$.
%\end{remark}

\begin{Example}\label{FedericoExample}
We illustrate   Lemma \ref{FedericoLemma}  by computing  the degree of the canonical map of the first surface  in our main Theorem \ref{MainTheo} again.
Recall that the mobile part of the canonical system  is generated by: 
\[
D_1:= 3F_2+4G_2, \quad 
D_2 :=F_1+2F_2+3G_0+G_1 \quad \makebox{and} \quad 
D_3:= F_0+2F_1+G_0+G_1+2G_2. 
\]
We determine the correction term to the self intersection number for each of the four base-points: 
\[
p_{12}, \quad p_{20}, \quad p_{21} \quad \makebox{and} \quad  p_{22}, \qquad \makebox{where} \quad \lbrace p_{ij} \rbrace =F_i\cap G_j. 
\] 
\begin{itemize}
\item
Around  $p_{12}$, the divisors $D_i$ are given by
$4G_2$, $F_1$   and $2F_1+2G_2$. In the notation of the Lemma 
$a=4$, $b=1$ and $c=d=2$. This implies  $c+md=10\geq 4$ and $d\geq q-1=-1$.  The correction term is  $ab=4$. 
\item
Around $p_{20}$ the divisors are 
$3F_2$,   $2F_2+ 3G_0$  and $G_0$.   
In this case  $a=3, b=1, c=2, d=3$ and the correction term is $ab=3$. 
\item
Around  $p_{21}$, we have 
$3F_2$, $2F_1+ G_1$  and $G_1$, which  yields $3$ as correction term. 
\item
Around  $p_{22}$ we have 
$3F_2+4G_2$, $2F_2$   and $2G_2$, thus the  correction term is $4$. 
\end{itemize}
The degree of the canonical map is therefore given by 
\[
\deg(\Phi_{K_S})=(3F_2+4G_2)^2 - 4-3-3-4= 10. 
\]

\end{Example}

\noindent 
For completeness, we point out: 

\begin{rem}
Our surfaces in Theorem  \ref{MainTheo} are examples of \emph{Beauville surfaces of unmixed type}, because $F/\mathbb Z_7^2\simeq \mathbb P^1$ and the  quotient cover   $\pi \colon F\to \mathbb P^1$ is branched in three points, as we  explained above. 
Beauville surfaces are precisely  the \emph{rigid surfaces isogenous to a product}, i.e. those  that  allow no non-trivial deformations \cite{BGV15}.
While  we cannot classify all regular surfaces isogenous to a product with $p_g=3$, it is possible to classify 
those that arise under an \emph{unmixed diagonal action of an abelian group}, thanks to the MAGMA  algorithm  \cite{MAGMA} from the paper   \cite{GPR}.  In particular, we know all 
\emph{unmixed Beauville surfaces} $S$ with \emph{abelian group}
and $p_g=3$. They form   seven biholomorphism  classes, which  can all   
be  realized as quotients of a product of two Fermat septics modulo $\mathbb Z_7^2$.  
Three of 
the four examples  which are not in
 the table 
of Theorem \ref{MainTheo} have generically finite canonical maps of degree $5$, $7$ and $14$, whilst the canonical map of the fourth  surface  is composed with a pencil. 
The  degrees of the  $4$th and $6$th surfaces of the table can be easily obtained  applying Lemma \ref{FedericoLemma}, in analogy to our computation in Example \ref{FedericoExample}. Regarding the $5$th surface the assumptions of the Lemma \ref{FedericoLemma} are fulfilled for all base-points, except for 
 $F_1\cap G_2=\lbrace p_{12}\rbrace$. Here locally the configuration of the three divisors spanning $\vert M \vert $ is as  in Remark \ref{CorrTerDeg7}:
 \[
 3F_1, \quad 2G_2 \qquad \makebox{and} \qquad  F_1+G_2. 
 \]
 Thus the correction term of $M^2-\hat{M}^2$ coming from $p_{12}$ is $5$ and we obtain $\deg(\Phi_{K_{S}}) =7$ using Lemma \ref{FedericoLemma}   for the other base-points. \\ 
Extending the table  of Theorem \ref{MainTheo}, we have:
 
 \[
\begin{tabular}{c | c | c | c | c }
\makebox{No} & $A$ & \makebox{Basis of $H^0(K_{S})$}  & $\Phi_{K_{S}}(x,y)$ & $\deg(\Phi_{K_{S}})$  \\
\hline
4. & $ \begin{pmatrix} 3&3 \\ 6 & 2\end{pmatrix} $ & $
\lbrace \omega_{0203}, \omega_{1004}, \omega_{3112} \rbrace$ &   
$(x_0^2x_2^2y_0y_2^3: x_0^3x_1y_2^4:x_1^3x_2y_0y_1y_2^2)$ & $5$ \\ 
\hline
5. & $ \begin{pmatrix} 5&4 \\ 6 & 5\end{pmatrix} $ & $\lbrace \omega_{1022}, \omega_{2131}, \omega_{4010} \rbrace$ &   
$(x_0^3x_1y_1^2y_2^2:x_0x_1^2x_2y_1^3y_2: x_1^4y_0^3y_1)$ & $7$ \\ 
\hline
6. & $ \begin{pmatrix} 1&1 \\ 6 & 2\end{pmatrix} $ & 
$\lbrace  
\omega_{0101},
\omega_{1313},
\omega_{3030} \rbrace$  &  
$(x_0^3x_2y_0^3y_2: x_1x_2^3y_1y_2^3: x_0x_1^3y_0y_1^3)$ &  $14$ \\ 
\hline
7. & $ \begin{pmatrix} 2&2 \\ 6 & 3\end{pmatrix} $ & $\lbrace \omega_{0202}, \omega_{2121}, \omega_{4040} \rbrace$ &   
$(x_0^2x_2^2y_0^2y_2^2: x_0x_1^2x_2y_0y_1^2y_2: x_1^4y_1^4)$ & ${\rm im}(\Phi_{K_{S_7}})= \lbrace y^2=xz \rbrace \subset \mathbb P^2$ \\ 
\hline
\end{tabular}
\]

\end{rem}

\end{document}